\documentclass[reqno]{amsart}
\usepackage[foot]{amsaddr}

\usepackage{graphicx} %
\usepackage{amsmath,amssymb,amsthm}
\usepackage{mathtools}
\usepackage{algorithm}
\usepackage{algorithmic}
\usepackage{multicol}
\usepackage{fullpage}
\usepackage{hyperref}
\usepackage{xcolor}
\usepackage{mleftright}
\usepackage{enumitem}
\usepackage{caption}
\usepackage{float}
\usepackage{hhline}
\usepackage{tikz}
\usepackage{tikz-cd}
\usepackage{quiver} 
\usepackage[nameinlink]{cleveref}
\usepackage{refcount}

\newcommand{\qand}{\quad\text{and}\quad}

\newcommand{\R}{{\mathbb{R}}}

\newcommand{\N}{{\mathbb{N}}}

\newcommand{\wrt}{\,\textnormal d}
\newcommand{\deriv}[2]{\frac{\wrt^{#2}}{\wrt{#1}^{#2}}}

\newcommand{\wt}{\widetilde}

\newcommand{\abs}[1]{\mleft|#1\mright|}

\newcommand{\pare}[1]{\mleft(#1\mright)}

\newcommand{\set}[1]{{\left\{{#1}\right\}}}

\newcommand{\spliteq}[2]{\begin{equation}#1\begin{split}#2\end{split}\end{equation}}
\newcommand{\eq}[1]{\begin{equation}{#1}\end{equation}}

\DeclareMathOperator*{\E}{\mathbb{E}}

\DeclareMathOperator*{\argmin}{arg\,min}

\DeclareMathOperator{\sign}{sign}

\newtheorem{theorem}{Theorem}[section]
\newtheorem{lemma}[theorem]{Lemma}
\newtheorem{proposition}[theorem]{Proposition}

\AddToHook{env/lemma/begin}{\crefalias{theorem}{lemma}}
\crefname{theorem}{Theorem}{Theorems}
\crefname{lemma}{lemma}{lemmas}
\crefname{equation}{}{}

\makeatletter
\newtheorem*{rep@theorem}{\rep@title}
\newcommand{\newreptheorem}[2]{%
\newenvironment{rep#1}[1]{%
 \def\rep@title{#2 \ref{##1}}%
 \begin{rep@theorem}}%
 {\end{rep@theorem}}}
\makeatother

\newreptheorem{theorem}{Theorem}
\newreptheorem{coro}{Corollary}

\newtheorem{remark}{Remark}
\theoremstyle{definition}

\DeclareMathOperator{\law}{law}
\DeclareMathOperator{\medish}{middle}
\DeclareMathOperator{\mean}{mean}

\DeclareMathOperator{\Ber}{Ber}

\title[maximum-entropy median-martingale]{The maximum-entropy median-martingale}
\author{Rikhav Shah}
\address{Massachusetts Institute of Technology}
\email{rdshah@mit.edu}
\author{Vilas Winstein}
\address{University of California, Berkeley}
\email{vilas@berkeley.edu}

\date{May 4, 2026}

\begin{document}

\begin{abstract}
    This short note explores the maximum-entropy walk on the unit interval that is a \textit{median-martingale}. That is, the median of its next state is equal to its current state. The stationary distribution of this walk is the arcsine distribution, and we provide a proof that elucidates the connection to two classical arcsine laws for Brownian motion.
    The notion of a martingale is further generalized, and a larger class of walks is considered and similarly characterized.
\end{abstract}

\maketitle

\section{Introduction}
A martingale is a random process $\set{X_t}_{t\in\N}$ which satisfies the relationship that $\E(X_t|X_{t-1},\ldots,X_1)=X_{t-1}$. In this note, we consider a related definition in which the mean is replaced by the median:
\eq{\label{a0}\text{median}(X_t\,|\,X_{t-1},\ldots,X_1)=X_{t-1}.}
Unlike the mean, the median is not always unique; that is, for a random variable $X$ there may be multiple values $m$ such that $\Pr(X\le m)=\Pr(X\ge m)$ (e.g. $X$ is a Bernoulli random variable). When we write $\text{median}(X)=m$, we mean $m$ is a value satisfying $\Pr(X\le m)=\Pr(X\ge m)$.
When $\set{X_t}$ satisfies \cref{a0}, we call it a ``median-martingale''.
Though this definition is natural and has appeared before \cite{b0,b1,b2}, it does not seem to have garnered significant attention.

This note considers the median martingale which takes values in the unit interval $(0,1)$ and maximizes the differential entropy of each next state. More explicitly, given a current state $X_{t}=x$, the differential entropy of the next state $X_{t+1}$ is maximized subject to \cref{a0} for the mixture
\eq{\label{a1}\frac12\text{Unif}(0,x)+\frac12\text{Unif}(x,1).}
We therefore take $\set{X_t}_{t\in\N}$ to be the Markov chain where $(X_t\,|\,X_{t-1}=x)$ is distributed according to \cref{a1}.
Notice a key way this departs from regular martingales: Doob's martingale convergence theorem tells us that if a martingale is bounded, then it must converge. On the other hand, the process we just described is bounded but clearly does not converge: for any given current state, the density of the next state is lower bounded by 1/2 on the entire interval.

This process has been studied before in the context of iterated function systems and attractor sets.
In particular, Diaconis and Freedman \cite{b3} show for any starting point $X_0$ that $\law(X_t)$ converges to a stationary distribution in the Prokhorov metric as $t\to\infty$.
They show how to derive a formula for this stationary distribution which a reader can straightforwardly verify. This stationary distribution is the \textit{arcsine distribution}, which is famous for the so-called \textit{arcsine laws}, which catalog its appearance in statistics of Brownian motion \cite{b4}.
Stoyanov and Pirinsky \cite{b5} show the stronger statement that for any starting point $X_0$, $\law(X_t)$ converges to the arcsine distrubtion in total variation distance. Both \cite{b3,b5} and others (cf. \cite{b6}) study several interesting generalizations of this walk. We focus on the appearance of the arcsine distribution and this martingale-like property, which leads us in a different direction.
For completeness, we include the calculation referred to in \cite{b3} which verifies that the arcsine distribution is stationary.

Our main contribution is an investigation into the appearance of the arcsine law in this setting. In particular, we provide a new proof that the stationary distribution of $\set{X_t}_{t\in\N}$ is the arcsine law, and our proof makes explicit a relationship of $\set{X_t}_{t\in\N}$ to Brownian motion.
We introduce a process on functions on $[0,1]$, each state of which is an instantiation of Brownian motion. The positive occupation time of each state, which is arcsine distributed by the arcsine laws, is shown to have the same stationary distribution as $\set{X_t}_{t\in\N}$.

We further investigate a generalization of this walk, which is distinct from the generalizations considered in previous works. Our generalization is related to two observations: first, one may replace the median in \cref{a0}  with yet other representatives of ``the middle'' of a distribution. Second, the denominator of the arcsine density, $\sqrt{x(1-x)}$, is the \textit{geometric mean} of $x$ and $1-x$, and one obtains an interesting family of distributions by replacing this this with other means, e.g. arithmetic, quadratic, etc).
Our contribution is primarily the suggestion of this class of walks as an interesting object of study. Our analysis of it in one case is an adaptation of the ``differentiation under the integral sign'' technique used by \cite{b3} for the median-martingale.

\section{Main analysis}
The walk $\set{X_t}_{t\in\N}$ can be described as
\eq{\label{a2}
\pare{X_t\,|\,X_{t-1}=x}\text{ is uniform on }(0,x)\text{ or uniform on }(x,1)\text{ with equal probability.}
}
The transition kernel on $(0,1)^2$ can be written as
\[K_X(s,x)=\frac1{2x}\mathbf1_{0<s < x}+\frac1{2(1-x)}\mathbf1_{x<s<1}.\]
The first argument $s$ always represents the next state, and $x$ the current state so that $K_Xf$ is the density of $X_{t+1}$ when $f$ is the density of $X_t$.
We begin with a brief proof of the main theorem which simply performs an explicit integral calculation.
\begin{theorem}\label{a3}
The stationary distribution of $\set{X_t}_{t\in\N}$ defined by \cref{a2} is the arcsine distribution.
\end{theorem}
\begin{proof}Let $\rho$ be the arcsine density function,
\(\rho(s)=\frac1\pi\frac1{\sqrt{s(1-s)}}\mathbf1_{s\in(0,1)}.\)
Since $K_X$ is lower bounded by 1/2 on the entire domain, Doeblin minorization implies the stationary distribution is unique. It thus suffices to show that $\rho$ solves the equation \(K_X\rho=\rho.\) To this end,
\spliteq{}{
(K_X\rho)(s)
=\int_0^1K_X(s,x)\rho(x)\wrt x
&=\frac1{2\pi}\int_0^s\frac1{(1-x)^{3/2}x^{1/2}}\wrt x
+\frac1{2\pi}\int_s^1\frac1{(1-x)^{1/2}x^{3/2}}\wrt x
\\&=\frac1{\pi}\sqrt{\frac s{1-s}}
+\frac1{\pi}\sqrt{\frac {1-s}{s}}
=\rho(s).}
\end{proof}
We now turn our attention to a proof of \Cref{a3} that draws an explicit connection to Brownian motion and the classical arcsine laws.
Our proof relates $\set{X_t}_{t\in\N}$ to a new process $\set{Y_t}_{t\in\N}$, which is not a median-martingale, defined by a similar rule,
\eq{\label{a4}
\pare{Y_t\,|\,Y_{t-1}=x}\text{ is uniform on }(0,x)\text{ or uniform on }(1-x,1)\text{ with equal probability.}}
We will realize $Y_t$ as a \textit{hidden Markov process}, where the underlying chain is a sequence of continuous functions on the interval. In the underlying chain, the stationary distribution will be Brownian motion. The essential idea is to use two arcsine laws in concert: the positive occupation time of one Brownian motion in the process will become the location of the last zero in the next Brownian motion.

Let $W_t(\cdot)\in C^0([0,1])$ denote the state of the underlying chain at time $t\in\N$.
Set $\wt Y_t$ to be the positive occupation time of $W_t(\cdot)$,
\eq{\label{a5}
\wt Y_t=\int_0^1\mathbf1_{W_t(s)>0}\wrt s
.}
Given $W_{t-1}(\cdot)$, sample $W_t(\cdot)$ as follows. Sample a Brownian bridge
$B:[0,1]\to\R$ and Brownian motion conditioned to be positive $B_+:[0,1]\to\R$.
Note we will often refer to $B_+$ as a \emph{Brownian meander}, following \cite{b7}.
Let $R$ be an independent Rademacher random variable.
We take $B, B_+$, and $R$ to be independent of one another and of $W_{t-1}$.
Now set
\[
W_t(s)=\begin{cases}
   \sqrt{\wt Y_{t-1}}B\pare{\frac s{\wt Y_{t-1}}} & s<\wt Y_{t-1}\\
   \sqrt{1-\wt Y_{t-1}}RB_+\pare{\frac{s-\wt Y_{t-1}}{1-\wt Y_{t-1}}}& s > \wt Y_{t-1}.\end{cases}
\]
Our first claim is that this hidden Markov process $\wt Y_t$ \cref{a5} coincides with the simple Markov chain $Y_t$ \cref{a4}.
\begin{lemma}\label{a6}
For every $t\ge2$, there is a coupling such that
\[Y_{t-1}=\wt Y_{t-1}\implies Y_t=\wt Y_t\text{ a.s.}\]
\end{lemma}
\begin{proof}
Observe that
\spliteq{}{
\wt Y_t
=\int_0^{\wt Y_{t-1}} \mathbf1_{W_t(s)>0}\wrt s
+\int_{\wt Y_{t-1}}^1 \mathbf1_{W_t(s)>0}\wrt s
=\wt Y_{t-1}\int_0^1 \mathbf1_{B(s)>0}\wrt s
+\frac{R+1}2\pare{1-\wt Y_{t-1}}
.}
The integral $\int_0^1 \mathbf1_{B(s)>0}\wrt s$ is the positive occupation time of a Brownian bridge. In the same work proving the arcsine laws, L\'evy shows this is uniformly distributed on the unit interval \cite{b4}, cf. \cite{b8}.
When $R=-1$, $\wt Y_t$ is a uniform sample from $(0,\wt Y_{t-1})$. When $R=1$, $\wt Y_t$ is a uniform sample from $(1-\wt Y_{t-1},1)$. In particular, the transition kernel for $Y_t$ coincides with the transition kernel for $\wt Y_t$.
\end{proof}

We next argue that the stationary distributions of $\set{Y_t}_{t\in\N}$ and $\set{X_t}_{t\in\N}$ coincide.
\begin{lemma}\label{a7}
    The stationary distributions of $\set{Y_t}_{t\in\N}$ and $\set{X_t}_{t\in\N}$ coincide.
\end{lemma}
\begin{proof}
Both walks can be described in terms of kernel functions. Let
\[
K_X(s,x)=
\frac1{2x}\mathbf1_{\set{s<x}}
+\frac1{2(1-x)}\mathbf1_{\set{s>x}}
,\quad\quad
K_Y(s,x)=
\frac1{2x}\mathbf1_{\set{s<x}}
+\frac1{2x}\mathbf1_{\set{s>1-x}}
\]
These kernels carry the distribution at time $t$ to the distribution at time $t+1$. For $Z\in\set{X,Y}$, let $\rho_Z$ be the density of the stationary distribution, i.e. a solution to
\(\rho_Z(s)=\int K_Z(s,x)\rho_Z(x)\wrt x.\)
Uniqueness of $\rho_X$ follows from Doeblin minorization: $K_X(s,x)\ge\frac12$ for all $s,x\in(0,1)$. We claim that $\rho_X$ and $\rho_Y$ are symmetric around $\frac12$. To see that $\rho_Y$ is symmetric, observe that $K_Y(\cdot,b)$ is symmetric for each $x$, and therefore the range of $K_Y$ is functions which are symmetric. To see that $\rho_X$ is symmetric, observe that $K_X(s,x)=K_X(1-s,1-x)$, so $\rho_X(1-\cdot)$ is also stationary.
A consequence of this claim is that $\rho_X$ and $\rho_Y$ both satisfy the
following symmetrized version of the stationary equation,
\[\rho_Z(s)=\int\frac{K_Z(s,x)+K_Z(s,1-x)}2\rho_Z(x)\wrt x.\]
Now notice that
\(K_X(s,x)+K_X(s,1-x)=K_Y(s,x)+K_Y(s,1-x)\)
so this equation is the same for $Z=X$ and $Z=Y$. Again by Doeblin minorization, $K_X(s,x)\ge\frac12$ for all $s,x\in(0,1)$ implies the solution is unique, so $\rho_X=\rho_Y$.
\end{proof}

Our final step is to show that the stationary distribution of $\wt Y_t$ is the arcsine distribution.
An arcsine law tells us that if $W_t(\cdot)$ is Brownian motion, then $\wt Y_t$ is arcsine-distributed.
The final piece of the puzzle is that the stationary distribution of the chain $\set{W_t}_{t\in\N}$ is Brownian motion. This will also follow from an arcsine law: If $W_t(\cdot)$ is Brownian motion, its last zero is arcsine-distributed.    

\begin{lemma}\label{a8}
For each $t\ge2$, if $W_{t-1}(\cdot)$ is Brownian motion then $W_t(\cdot)$ is Brownian motion.
\end{lemma}
\begin{proof}
Let $W(\cdot)$ be standard Brownian motion on $[0,1]$ and $Y$
its last zero.
By \cite[Chapter XII, Exercise (3.8)]{b7},
\begin{equation}
    t \mapsto B(t) = \sqrt{Y} W\pare{\tfrac{t}{Y}},
    \qquad Y, \qquad \text{and} \qquad
    t \mapsto B_+(t) = \sqrt{1 - Y} \abs{W\pare{\tfrac{t-Y}{1-Y}}}
\end{equation}
are independent.
Moreover, $B$ is a Brownian bridge, $Y$ is an arcsine-distributed random variable,
and $B_+$ is a Brownian meander, i.e.\ a Brownian motion on $[0,1]$ conditioned to be positive.
Furthermore, one can recover $W$ from $B,Y,B_+$, and  
$R = \sign(W(1))$ which is a Rademacher variable independent from $B,Y,B_+$.

Now since $\wt Y_{t-1}$ is arcsine distributed and independent from $B, B_+$, and $R$, we may
substitute it for $Y$ in the above inverse construction of $W$ and we again obtain
a Brownian motion.
This construction is the definition of $W_t(\cdot)$, so this finishes the proof.
\end{proof}

With these lemmas in place, we're now ready to give an independent proof that the arcsine distribution is stationary.

\begin{proof}[Second proof of \Cref{a3}]
Let $W_1$ be Brownian motion. Then $W_t(\cdot)$ is Brownian motion for all $t$ by \Cref{a8}. Then $\wt Y_t$ is arcsine-distributed for all $t$ by an arcsine law. Then $Y_t$ is arcsine-distributed for all $t$ by \Cref{a6}. In particular, the arcsine distribution is the stationary distribution of the Markov chain $Y_t$, which coincides with the stationary distribution of $X_t$ by \Cref{a7}.
\end{proof}

\section{Generalizations}
In this section, we generalize \cref{a2} to a one-parameter family of random walks, of which \cref{a2} is a ``tipping point'' with different behavior on either side. This family is indexed by a parameter $p\in\R$ and is defined by the transition rule
\eq{\label{a9}
\pare{X_t\,|\,X_{t-1}=x}=\begin{cases}
    \text{uniform on }(0,x)\text{ with probability proportional to }{x^p}\\
    \text{uniform on }(x,1)\text{ with probability proportional to }{(1-x)^p}\\
\end{cases}.}
When $p<0$, these probabilities blow up at the end points $x\in\set{0,1}$, so it's natural to extend the definition of the chain to include $\set{0,1}$ as absorbing states when $p<0$. We state the consequence of this in \Cref{a10}.
We can be more precise in the specification of the chain by describing its transition kernel,
\spliteq{\label{a11}}{
K_p(s,x)
&=\frac{\Pr(\text{jump towards }0)}x\mathbf1_{0<s<x}
+\frac{\Pr(\text{jump towards }0)}{1-x}\mathbf1_{x<s<1}
\\&=\frac{x^{p-1}}{x^p+(1-x)^p}\mathbf1_{0<s<x}
+\frac{(1-x)^{p-1}}{x^p+(1-x)^p}\mathbf1_{x<s<1}
.}
When $p=0$, this is just the original walk \cref{a2}. The chain for $p<0$ exhibits two interesting phenomena. First, it will \textit{not} approach a stationary distribution in total variation distance. Second, it \textit{will} exhibit a property similar to being a martingale and a median martingale, which we describe shortly. On the other hand, the chain for $p>0$ exhibits the opposite phenomena. It \textit{will} approach a stationary distribution, but does not exhibit this martingale-like property.
Thus, $p=0$ is a ``sweet spot'' enjoying both a martingale-like property and strong convergence to a stationary distribution. 

The generalization of martingales and median-martingales we call $\ell_q$-minimizing-martingales. Recall that the mean and median can be expressed as the $\ell_2$ and $\ell_1$ minimizers respectively of the expected distance to the random variable, i.e.
\[
\E(X)=\argmin_z\E\abs{z-X}^2,\quad
\text{median}(X)=\argmin_z\E\abs{z-X}.\]
We can generalize this to a class of $\ell_q$-minimizers for any $q>0$,
\[\medish_q(X)=\argmin_z\E\abs{z-X}^q.\]
An $\ell_q$-minimizing-martingale is then a process $\set{X_t}_{t\in\N}$ such that
\eq{\label{a12}
\medish_q(X_t\,|\,X_{t-1},\ldots X_1)=X_{t-1}.}
When $q>1$, the functional $f_X(z)=\E\abs{z-X}^q$ is differentiable and convex, and therefore the minimizer is the unique solution to $f'_X(z)=0$. One can directly compute for each $x\in(0,1)$ that $f'_{X_t|X_{t-1}=x}(x)=0$ if $p=1-q$, which results in the following proposition.
\begin{proposition}
    The process \cref{a9} is an $\ell_q$-minimizing martingale if $q\ge1$ and $p=1-q$.
\end{proposition}
\begin{proof}
    The case of $q=1$, $p=0$ is exactly the median-martingale already considered so we proceed to the $q>1$ case. Condition on $X_1=x$ and let $f(z)=\E\abs{z-X_2}^q$. Take its derivative and evaluate at $z=x$,
    \spliteq{}{
    f'(x)
    &=q\E\pare{\abs{x-X_2}^{q-1}\sign(z-X_2)  }
    \\&=q\int_0^1K_p(s,x)\abs{x-X_2}^{q-1}\sign(x-X_2)\wrt s
    \\&=\frac q{x^p+(1-x)^p}\pare{x^{p-1}\int_0^x\abs{x-s}^{q-1}\wrt s-(1-x)^{p-1}\int_x^1\abs{x-s}^{q-1}\wrt s}
    \\&=\frac1{x^p+(1-x)^p}\pare{x^{p+q-1}-(1-x)^{p+q-1}}
    }
    which indeed vanishes when $p+q=1$. $f$ is the average of differentiable and convex functions, so the zero of the derivative is the unique minimizer.
\end{proof}
\begin{remark}
    Notice this proposition only applies for $p\le0$. 
One may compute that $f'_{X_t|X_{t-1}=x}(x)=0$ for $q<1$ as well, but $f_{X_t|X_{t-1}=x}(\cdot)$ is no longer convex in this case and $x$ may very well be point of inflection, not a minimum, of $f_{X_t|X_{t-1}=x}(\cdot)$. We therefore do not obtain a similar property for $p>0$. 
\end{remark}

We now turn our attention to the stationary distribution. Our derivation for the density function for $p>0$ follows a technique presented in \cite{b3}.

\begin{theorem}\label{a13}
For any $p>0$, the walk \cref{a9} converges to the stationary distribution with density
\[\rho(x)=\frac1{Z_p}\pare{\frac2{x^p + (1-x)^p}}^{1/p}\cdot\mathbf1_{x\in(0,1)}\]
where $Z_p\in(2\log 2,\pi)$ is a constant depending on $p$. 
\end{theorem}
\begin{proof}
Fix $p>0$. If $p\le1$, then $K_p(s,x)\ge\frac12$ everywhere. If $p\ge1$, then $K_p(s,x)\ge\frac12$ for  $s\in(1/2-1/(10p),1/2+1/(10p))$ and all $x$. This establishes convergence by Doeblin minorization.
Suppose the existence of a positive, differentiable density $\rho$ on $(0,1)$ which is stationary. 
Take the derivative of both sides of the equation $\rho(s)=\int_0^1K(s,x)\rho(x)\wrt x$ to obtain
\spliteq{\label{a14}}{
\deriv s{}\rho(s)
&=\deriv s{}\pare{\int_s^1\frac{x^{p-1}}{x^p+(1-x)^p}\rho(x)\wrt x
+\int_0^s\frac{(1-x)^{p-1}}{x^p+(1-x)^p}\rho(x)\wrt x}
\\&=-\frac{s^{p-1}}{s^p+(1-s)^p}\rho(s)
+\frac{(1-s)^{p-1}}{s^p+(1-s)^p}\rho(s)
\\&=-\frac{s^{p-1}-(1-s)^{p-1}}{s^p+(1-s)^p}\rho(s).}
Dividing both sides by $\rho(s)$ gives
\[\deriv s{}\log\rho(s)
=-\frac1p\deriv s{}\log(s^p+(1-s)^p)
.\]
Integrate and take the exponential to obtain
\eq{\label{a15}\rho(s)=\frac C{(s^p+(1-s)^p)^{1/p}}.}
for some $C$ depending on $p$. Notice that this proves uniqueness: if $\rho$ is stationary, then the above steps show it must have the form \cref{a15}.
We now argue that $\rho$ is integrable and compute its normalization constant. 
First, insert a factor of $2^{1/p}$ so that we may express
\[
\rho(s)=\frac1{Z_p}\frac1{\mean_p(s,1-s)}
\]
where $\mean_p(a,b)=\pare{\frac{a^p+b^p}2}^{1/p}$ is known as the $\ell_p$-mean of the numbers $a$ and $b$, and $Z_p$ is the remaining factor of the normalization constant required to make $\int_0^1\rho(s)\wrt s=1$.
By the power-mean inequality, $\mean_p(a,b)$ is non-decreasing in $p$ and strictly increasing for $a\neq b$, and therefore $Z_p$ is strictly decreasing. By the observation that
\[\lim_{p\to0}\mean_p(a,b)=\sqrt{ab}\qand\lim_{p\to\infty}\mean_p(a,b)=\max(a,b),\]
we obtain finite, positive upper and lower bounds on $Z_p$ by computing $Z_0=\pi$ and $Z_\infty=2\log2$. This establishes the integrability of $\rho$.

Finally, using \cref{a14}, one verifies that $\rho$ indeed solves the differential equation $\rho'(s)=\deriv s{}\int_0^1K(s,x)\rho(x)\wrt x$ so we have $\rho(s)=\int_0^1K(s,x)\rho(x)\wrt x+c$ for some constant $c$, but since $K$ is stochastic, integrating out $s$ reveals $c=0$, so $\rho$ is indeed stationary.
\end{proof}
\begin{remark}\label{a10}
The above proof did not use that $p>0$ until the last step, in the computation of $Z_p$.
If $p<0$, no such $Z_p$ making $\int_0^1\rho(s)\wrt s=1$ exists. In particular, $\frac1{\mean_p(s,1-s)}$ is not integrable. To see this, compute
\[
\lim_{s\to0}\frac{\mean_p(s,1-s)}{s}
=\lim_{s\to0}\pare{\frac12+\frac{(1-s)^p}{2s^p}}^{1/p}=\frac12\]
so $\frac1{\mean_p(s,1-s)}$ blows up as $O(1/s)$ for small $s$. 
For $p<0$, we can naturally extend the definition \cref{a9} to include 0 and 1 as absorbing states. In this case, any Bernoulli random variable is a stationary distribution, and no continuous stationary distribution exists. For the walk's convergence in this case, see \Cref{a16}
\end{remark}
\begin{remark}
    Notice that $\rho(x)$ is proportional to the reciprocal of the $\ell_p$-mean of the distances from $x$ to each of the endpoints of the interval. In the limit of $p\to0$, this is the \textnormal{geometric mean} of those distances, $\sqrt{x(1-x)}$, which makes $\rho$ exactly the arcsine distribution. In particular, \cref{a13} recovers \cref{a3}.
\end{remark}
\begin{remark}
    There is a transition at $p=1$, where $\rho$ is uniform.
    For $p<1$, $\rho$ is bimodal with peaks at 0 and 1, and for $p>1$, $\rho$ is unimodal with a peak at 1/2.
    In the limit of $p\to\infty$, the density is $\rho(x)\propto\frac1{\max(x,1-x)}$. This is the chain that at state $x$, picks the further endpoint of the unit interval and jumps a uniform amount towards it.
\end{remark}

\begin{theorem}\label{a16}
For every $p<0$, $\set{X_t}_{t\in\N}$ given by \cref{a9} almost surely converges to
\[
\Ber\pare{
\frac12-\frac{x^{p-1}-\left(1-x\right)^{p-1}}{2\left(x^{p}+\left(1-x\right)^{p}\right)^{\frac{p-1}{p}}}
}
\]
\end{theorem}
\begin{proof}
Define the continuous bijection $f:[0,1]\to[-1,1]$,
\[f(x)=-\frac{x^{p-1}-\left(1-x\right)^{p-1}}{\left(x^{p}+\left(1-x\right)^{p}\right)^{\frac{p-1}{p}}}=-\deriv x{}\pare{x^p+(1-x)^p}^{1/p}.\]
Direct integration shows that $f(X_t)$ is a martingale. By Doob's martingale convergence theorem, $X_t$ almost surely converges. Since there are no absorbing states in the open interval, $(0,1)$, we must have $\lim_{t\to\infty} X_t\in\set{0,1}$ almost surely. In particular, $(X_t|X_0=x)\to\Ber\pare{\frac{f(x)+1}2}$.
\end{proof}

\section*{Acknowledgments}
The authors thank Elchanan Mossel and Jordan Stoyanov for comments on an earlier version of this note, Adam Jaffe for an engaging discussion about this walk many years ago, and anonymous reviewers for helpful revisions.

\bibliographystyle{acm}
\bibliography{outbib}

\end{document}